\pdfoutput=1
\documentclass[a4paper, 11pt]{amsart}
\usepackage[top=1.27in, bottom=1.27in, left=1.28in, right=1.28in]{geometry}
\usepackage{graphicx}
\usepackage{amsfonts}
\usepackage{epsf}
\usepackage{amssymb}
\usepackage{amsmath}
\usepackage{amscd}
\usepackage{amsthm}
\usepackage{mathtools}
\usepackage{tikz}
\usetikzlibrary{
  cd,
  calc,
  positioning,
  arrows,
  decorations.pathreplacing,
  decorations.markings,
}
\usepackage{pdfpages}
\usepackage{setspace}
\usepackage{hyperref}
\usepackage[all]{xy}
\usetikzlibrary{matrix}
\usepackage{verbatim}
\usepackage{enumerate}
\usepackage{mathrsfs}
\usepackage{pinlabel}
\usepackage{lscape}
\usepackage{color}
\usepackage{enumitem}

\newtheorem{theorem}{Theorem}[section]

\newtheorem{proposition}[theorem]{Proposition}
\newtheorem{corollary}[theorem]{Corollary}
\newtheorem{lemma}[theorem]{Lemma}

\newtheorem{question}[theorem]{Question}
\newtheorem*{warning}{Warning}

\newenvironment{Warning}{\begin{warning}\rm}{\end{warning}}

\theoremstyle{definition}

\newtheorem{example}[theorem]{Example}

\newtheorem*{case2'}{Case 2$'$}
\newtheorem{theorem-named}{}
\newtheorem{theorem-labeled}{Theorem}
\newtheorem{definition-named}{}
\newtheorem{conjecture-named}{}
\newtheorem{case-named}{}

\numberwithin{equation}{section}

\newcommand{\R}{\mathbb{R}}
\newcommand{\C}{\mathbb{C}}

\def\Z{\mathbb{Z}}

\def\R{\mathbb{R}}

\def\C{\mathbb{C}}

\def\et{\quad\mbox{and}\quad}
\makeatletter
\usepackage{epsf,graphics,amssymb,amsmath,xcolor,asymptote,microtype}
\usepackage{hyperref}
\usepackage[justification=normal]{caption}
\usepackage[nameinlink,nosort]{cleveref}

\begin{asydef}
real strokewidth = 0.7;
transform myt = scale(1, 1.66);
pen braidpen = linewidth(strokewidth);

void mydraw(path p) {
    draw(myt * p, braidpen);
}

// x = some crossing
// 0 = the crossing in question
// o = no crossing

// do
// .....
// ...o.
// ..0..
// .x...
// .....

// except
// .x...
// ..x..
// ..0..
// .x...
// .....

// do
// ..x..
// ...x.
// ...x.
// ..0..
// .x...
// .....

// do
// .....
// .o...
// ..0..
// ...x.
// .....
real howFarOnTop(int[][] B, int x, int y) {
     if (( ((B[x + 1][y] == 0) && (B[x][y - 1] != 0) && (B[x - 1][y - 2] != 0) && ((B[x - 1][y + 1] == 0) || (B[x][y] == 0)))  || ((B[x + 1][y] != 0) && (B[x][y - 1] != 0) && (B[x - 1][y - 2] != 0) && (B[x + 1][y + 1] != 0) && (B[x][y + 2] != 0))) ||
         ( ((B[x][y] == 0) && (B[x + 1][y - 1] != 0) && (B[x + 2][y - 2] != 0) && ((B[x + 2][y + 1] == 0) || (B[x + 1][y] == 0)))  || ((B[x][y] != 0) && (B[x + 1][y - 1] != 0) && (B[x + 2][y - 2] != 0) && (B[x][y + 1] != 0) && (B[x + 1][y + 2] != 0))))
     {
        return 0.15;
        dot(myt * (x,y), scale(1.5) * green);
    }
    return 0;
}

real howFarOnBottom(int[][] B, int x, int y) {
     if (( ((B[x + 1][y - 1] == 0) && (B[x][y] != 0) && (B[x - 1][y + 1] != 0) && ((B[x - 1][y - 2] == 0) || (B[x][y - 1] == 0)))  || ((B[x + 1][y - 1] != 0) && (B[x][y] != 0) && (B[x - 1][y + 1] != 0) && (B[x + 1][y - 2] != 0) && (B[x][y - 3] != 0))) ||
         ( ((B[x][y - 1] == 0) && (B[x + 1][y] != 0) && (B[x + 2][y + 1] != 0) && ((B[x + 2][y - 2] == 0) || (B[x + 1][y - 1] == 0)))  || ((B[x][y - 1] != 0) && (B[x + 1][y] != 0) && (B[x + 2][y + 1] != 0) && (B[x][y - 2] != 0) && (B[x + 1][y - 3] != 0))))
        return -0.15;
    return 0;
}

path hidingDisk(int x, int y) {
    return shift(x - .5, y + .5) * scale(0.2) * inverse(myt) * unitcircle;
}

void drawToLeft(int[][] B, int x, int y, pen braidpen, bool gap) {
    pair outgoingAngle = ((B[x + 1][y - 1] == 0) || ((B[x + 1][y - 1] != 0) && (B[x][y + 1] != 0) && (B[x + 1][y + 2] != 0)) || ((B[x + 1][y - 1] != 0) && (B[x + 1][y - 2] != 0) && (B[x][y - 3] != 0))) ? up : up + left;
    pair incomingAngle = ((B[x - 1][y + 1] == 0) || ((B[x - 1][y + 1] != 0) && (B[x - 1][y + 2] != 0) && (B[x][y + 3] != 0)) || ((B[x - 1][y + 1] != 0) && (B[x][y - 1] != 0) && (B[x - 1][y - 2] != 0))) ? up : up + left;

    pair midpoint = (x - .5, y + .5);
    path p;
    if (howFarOnBottom(B, x, y) != 0)
        p = (x,y + howFarOnBottom(B, x, y) + howFarOnTop(B, x, y)){outgoingAngle}..(x - .2, y + .2)..midpoint..{incomingAngle}(x - 1, y + 1 + howFarOnTop(B, x - 1, y + 1) + howFarOnBottom(B, x - 1, y + 1));
    else if (howFarOnTop(B, x - 1, y + 1) != 0)
        p = (x,y + howFarOnBottom(B, x, y) + howFarOnTop(B, x, y)){outgoingAngle}..midpoint..(x - .8, y + .8)..{incomingAngle}(x - 1, y + 1 + howFarOnTop(B, x - 1, y + 1) + howFarOnBottom(B, x - 1, y + 1));
    else
        p = (x,y + howFarOnBottom(B, x, y) + howFarOnTop(B, x, y)){outgoingAngle}..midpoint..{incomingAngle}(x - 1, y + 1 + howFarOnTop(B, x - 1, y + 1) + howFarOnBottom(B, x - 1, y + 1));
    if (gap) {
        real[][] i = intersections(p, hidingDisk(x, y));
        draw(myt * subpath(p, 0, i[0][0]), braidpen);
        draw(myt * subpath(p, i[i.length - 1][0], 3), braidpen);
    } else
        draw(myt * p, braidpen);
}

void drawToRight(int[][] B, int x, int y, pen braidpen, bool gap) {
    pair outgoingAngle = ((B[x][y - 1] == 0) || ((B[x][y - 1] != 0) && (B[x + 1][y + 1] != 0) && (B[x][y + 2] != 0)) || ((B[x][y - 1] != 0) && (B[x][y - 2] != 0) && (B[x + 1][y - 3] != 0))) ? up : up + right;
    pair incomingAngle = ((B[x + 2][y + 1] == 0) || ((B[x + 2][y + 1] != 0) && (B[x + 2][y + 2] != 0) && (B[x + 1][y + 3] != 0)) || ((B[x + 2][y + 1] != 0) && (B[x + 1][y - 1] != 0) && (B[x + 2][y - 2] != 0))) ? up : up + right;

    pair midpoint = (x + .5, y + .5);
    path p;
    if (howFarOnBottom(B, x, y) != 0)
        p = (x,y + howFarOnBottom(B, x, y) + howFarOnTop(B, x, y)){outgoingAngle}..(x + .2, y + .2)..midpoint..{incomingAngle}(x + 1, y + 1 + howFarOnTop(B, x + 1, y + 1) + howFarOnBottom(B, x + 1, y + 1));
    else if (howFarOnTop(B, x + 1, y + 1) != 0)
        p = (x,y + howFarOnBottom(B, x, y) + howFarOnTop(B, x, y)){outgoingAngle}..midpoint..(x + .8, y + .8)..{incomingAngle}(x + 1, y + 1 + howFarOnTop(B, x + 1, y + 1) + howFarOnBottom(B, x + 1, y + 1));
    else
        p = (x,y + howFarOnBottom(B, x, y) + howFarOnTop(B, x, y)){outgoingAngle}..midpoint..{incomingAngle}(x + 1, y + 1 + howFarOnTop(B, x + 1, y + 1) + howFarOnBottom(B, x + 1, y + 1));
    if (gap) {
        real[][] i = intersections(p, hidingDisk(x + 1, y));
        draw(myt * subpath(p, 0, i[0][0]), braidpen);
        draw(myt * subpath(p, i[i.length - 1][0], 3), braidpen);
    } else
        draw(myt * p, braidpen);
}

void drawbraid(int[] b, int strands = -1, pen[] strandPens = {}) {
    // delete zeroes
    for (int i = 0; i < b.length;) {
        if (b[i] == 0)
            b.delete(i);
        else
            ++i;
    }

    // get right number of strands
    if (strands == -1) {
        int m = 1;
        for (int i = 0; i < b.length; ++i)
            if (abs(b[i]) > m)
                m = abs(b[i]);
        strands = m + 1;
    }

    // slide, prepare matrix
    int[][] B = array(strands + 3, array(b.length + 4, 0));
    for (int i = 0; i < b.length; ++i) {
        int x = abs(b[i]) + 1;
        int y = b.length + 3;
        while ((B[x - 1][y - 1] == 0) && (B[x][y - 1] == 0) && (B[x + 1][y - 1] == 0) && (y > 2))
            --y;
        B[x][y] = (b[i] > 0) ? 1 : -1;
    }

    // until where are there crossings?
    int highestY = 0;
    for (int y = 2; y < B[0].length; ++y) {
        bool thereAreCrossings = false;
        for (int x = 1; x < strands + 1; ++x)
            if (B[x][y] != 0) {
                thereAreCrossings = true;
                break;
            }
        if (! thereAreCrossings) {
            highestY = y;
            break;
        }
    }

    //grid
/*    for (int x = 1; x < strands + 1; ++x)
        draw(myt * ((x, 1)--(x, highestY)), lightgray);
    for (int y = 1; y < highestY; ++y)
        draw(myt * ((1, y)--(strands, y)), lightgray);
*/

    // follow the strands
    int[] strandNumbers;
    for (int x = 0; x < strands; ++x)
        strandNumbers.push(x);

    // draw
    for (int y = 2; y < highestY; ++y)
        for (int x = 1; x < strands + 1; ++x) {
            pen leftPen = braidpen;
            pen rightPen = braidpen;
            if (strandPens.length != 0) {
                leftPen = strandPens[strandNumbers[x - 1]];
                if (x != 1)
                    rightPen = strandPens[strandNumbers[x - 2]];
            }
            if (B[x][y] != 0) {
                drawToRight(B, x - 1, y, rightPen, (B[x][y] < 0));
                drawToLeft(B, x, y, leftPen, (B[x][y] > 0));

                int tmp = strandNumbers[x - 2];
                strandNumbers[x - 2] = strandNumbers[x - 1];
                strandNumbers[x - 1] = tmp;
            }
            else if (B[x + 1][y] == 0)
                draw(myt * ((x, y + howFarOnTop(B, x, y))--(x, y + 1 + howFarOnBottom(B, x, y + 1))), leftPen);
        }

    // cut off above and below
    clip(myt * ((0,2)--(2 + strands, 2)--(2 + strands, highestY)--(0, highestY)--cycle));
} 

void drawbraid(string s, int strands = -1, pen[] strandPens = {}) {
    int[] b = {};
    for(int i = 0; i < length(s); ++i) {
        int c = 0;
        int a = ascii(substr(s, i));
        if ((a >= 97) && (a <= 122))
            c = a - 96;
        else if ((a >= 65) && (a <= 90))
            c = -(a - 64);
        if (c != 0)
            b.push(c);
    }
    drawbraid(b, strands, strandPens);
}

//unitsize(1cm);
//drawbraid("abcddcbaabcd");
\end{asydef}

\begin{asydef}
unitsize(2.5mm);
strokewidth = 0.6;
\end{asydef}

\begin{document}


\title{ A note on the four-dimensional clasp number of knots}

\author{Peter Feller}
\address{ETH Zurich, R\"amistrasse 101, 8092 Zurich, Switzerland}
\email{peter.feller@math.ch}
\urladdr{people.math.ethz.ch/~pfeller/}
\author{JungHwan Park}
\address{Georgia Institute of Technology, Atlanta, GA, USA}
\email{junghwan.park@math.gatech.edu }
\urladdr{people.math.gatech.edu/~jpark929/}
\def\subjclassname{\textup{2020} Mathematics Subject Classification}
\expandafter\let\csname subjclassname@1991\endcsname=\subjclassname
\expandafter\let\csname subjclassname@2000\endcsname=\subjclassname


\begin{abstract}
Among the knots that are the connected sum of two torus knots with cobordism distance~1, we characterize those that have $4$-dimensional clasp number at least~2, and we show that their $n$-fold connected self-sum has $4$-dimensional clasp number at least~$2n$. Our proof works in the topological category.
To contrast this, we build a family of topologically slice knots for which the $n$-fold connected self-sum has $4$-ball genus $n$ and $4$-dimensional clasp number at least $2n$.
\end{abstract}
\maketitle

\section{Introduction}\label{sec:intro}
Recently the authors determined for which pairs of torus knots $\{ T_{p,q}, T_{p',q'}\}$ the cobordism distance is equal to 1, with one exception~\cite[Thm.~1.2]{FellerPark}. This was done by comparing explicit constructions of cobordisms with the lower bound for the cobordism distance using the $\nu^+$-invariant~\cite{Hom-Wu:2016-1} from the Heegaard Floer knot complex.
As an application, the authors determined which pairs of torus knots have Gordian distance $1$; see~\cite[Cor.~1.3]{FellerPark}. The first result of the present article has two motivations.

Firstly, when determining Gordian distance $1$ pairs of torus knots, we relied on $\nu^+$, but speculated that the proof could be done using the Tristram-Levine signatures~\cite{Tristram:1969-1,Levine:1969-1}; see \cite[Rmk.~4.3]{FellerPark}. Here, we partially confirm this speculation using signature calculations via the Hirzebruch-Brieskorn formula~\cite{Brieskorn_DifftopovonSing, GambaudoGhys_BraidsSignatures}.

Secondly, recent interest in the $4$-dimensional clasp number $c_4$ and its difference to the $4$-ball genus $g_4$~\cite{Kronheimer-Mrowka:2019-1,kronheimer-talk,Juhasz-Zemke:2020-1,Daemi-Scaduto:2020-1, daemi-scaduto-talk}, made us revisit the lower bounds we had on the Gordian distance of the pairs $\{ T_{p,q}, T_{p',q'}\}$ and consider whether they yield lower bounds on $c_4(K)$ and $c_4(\#^nK)$, where $K = T_{p,q} \# -T_{p',q'}$.
We see this article as shining a spotlight on what can be achieved using a somewhat classical setup of concordance homomorphism, while the new exciting invariants from~\cite{Kronheimer-Mrowka:2019-1,Juhasz-Zemke:2020-1,Daemi-Scaduto:2020-1} might be used to discover interesting phenomena beyond.

We write $-J$ to denote the reverse of the mirror image of a knot $J$, and we write $g_4^\mathrm{top}$ and $c_4^\mathrm{top}$ to denote the topological counterparts of $g_4$ and $c_4$, respectively.

\begin{theorem}\label{thm:claspnrvsg4fortorusknots} Let $\{ T_{p,q}, T_{p',q'}\}$ be a pair of positive torus knots and $K = T_{p,q} \# -T_{p',q'}$. 
If $g_4(K) =1$ and the Gordian distance between $T_{p,q}$ and $T_{p',q'}$ is not $1$, then
\[g_4(\#^nK) = g_4^\mathrm{top}(\#^nK) =n \et c_4(\#^nK)\geq c_4^\mathrm{top}(\#^nK)\geq 2n.\]
Furthermore, if $\{ T_{p,q}, T_{p',q'}\}$ is either $\{ T_{2,7}, T_{3,4}\}$, $\{ T_{2,9}, T_{3,5}\}$, or $\{ T_{3,7}, T_{4,5}\}$, then $$g_4(\#^nK) = g_4^\mathrm{top}(\#^nK) = n \et c_4(\#^nK)=c_4^\mathrm{top}(\#^nK)= 2n.$$ \end{theorem}

We comment on the condition $g_4(K)=1$.
The point is that $g_4(K)=1$ implies that $\{ T_{p,q}, T_{p',q'}\}$ must be one of an infinite family of pairs of torus knots, among which we know which are Gordian distance 1 apart, and, in fact, infinitely many are further apart~\cite{FellerPark}; see Section~\ref{sec:ProofviaSig}.
With this the first part of Theorem~\ref{thm:claspnrvsg4fortorusknots} reduces to showing $g_4^\mathrm{top}(\#^nK) \geq n$ and $c_4^\mathrm{top}(\#^nK)\geq 2n$ for an explicit infinite family of knots.
We do this in Section~\ref{sec:ProofviaSig} using Tristram-Levine signatures.
The `furthermore'-part comes down to finding appropriate crossing changes. For the pairs  $\{ T_{2,7},T_{3,4}\}$ and $\{ T_{2,9}, T_{3,5}\}$,\footnote{These pairs are familiar in singularity theory as the knots of singularities of the pairs of simple singularities $\{A_6,E_6\}$ and $\{A_8,E_8\}$, which have 
$\delta$-constant deformations to $A_4$ and $A_6$, respectively.} we believe Theorem~\ref{thm:claspnrvsg4fortorusknots} to be known to experts, but for the pair $\{ T_{3,7}, T_{4,5}\}$ we will exhibit two crossing changes turning one into the other explicitly.
\begin{question}
Are there other pairs of torus knots for which $g_4(K)=1$ and $c_4(K)=2$?
\end{question}

We also obtain further equivalent characterizations of Gordian distance 1 pairs of torus knots (compare with \cite[Cor.~1.3]{FellerPark}). Let $u_s(J)$ denote the slicing number of a knot $J$. We have the following chain of inequalities:
$$u_s(J) \geq c_4(J) \geq g_4(J).$$ It is unknown whether there is a knot $J$ with $u_s(J) > c_4(J)$; see e.g.~\cite[Prop.~6]{Owens-Strle:2016-1}.


\begin{corollary}If $\{ T_{p,q}, T_{p',q'}\}$ is a pair of positive torus knots, then the following statements are
equivalent.
\begin{enumerate}[font=\upshape]
\item\label{it:1cor14} The knots $T_{p,q}$ and $T_{p',q'}$ have Gordian distance $1$.
\item\label{it:2'cor14} The knot $T_{p,q}\# -T_{p',q'}$ has slicing number $1$.
\item\label{it:2cor14} The knot $T_{p,q}\# -T_{p',q'}$ has $4$-dimensional clasp number $1$.
\item\label{it:3cor14} The pair $\{ T_{p,q}, T_{p',q'}\}$ is one of the following:
\begin{enumerate}[font=\upshape]
\item $\{ T_{2,2n+1}, T_{2,2n+3}\}$ for $n\geq 0$,
\item $\{ T_{3,3n+1}, T_{3,3n+2}\}$ for $n\geq1$,
\item $\{ T_{2,5}, T_{3,4}\}$, $\{T_{2,7}, T_{3,5}\}$.\qed
\end{enumerate}
\end{enumerate}
\end{corollary}
\begin{question}\label{q:gordianvsclasps}Are there pairs of torus knots for which their Gordian distance is strictly larger than $c_4(\#^nK)/n$ for some integer $n\geq1$?\end{question}

Given that not just $c_4-g_4$ but also $c_4^\mathrm{top}-g_4$
can be large even on connected sums of a positive and a negative torus knot (by Theorem~\ref{thm:claspnrvsg4fortorusknots}), it seems interesting to find a family of examples of topologically slice knots on which $c_4-g_4$ is unbounded. Note that if $J$ is topologically slice (i.e.~$c_4^\mathrm{top}(J)=g_4^\mathrm{top}(J)=0$), then the Tristram-Levine signature 
for the $4$-dimensional clasp number vanishes; hence, one needs to use other invariants. Following in the footsteps of Livingston-Friedl-Zentner~\cite{LivingstonFriedlZentner} and Juh\'{a}sz-Zemke~\cite{Juhasz-Zemke:2020-1} among others, we use the Ozsv\'{a}th-Stipsicz-Szab\'{o} $\Upsilon$-invariant~\cite{OSS_2014} to find such families. 

Indeed,
let $D$ denote the positive untwisted Whitehead double of the right-handed trefoil, and let $J_{p,q}$ denote the
$(p, q)$-cable of a knot $J$, where $p$ is the longitudinal winding.

\begin{theorem}\label{thm:top} If $K_i = D_{2,2i+1} \# - T_{2,2i+1} \# -D$ where $i>1$, then $K_i$ is  a topologically slice knot that satisfies
\[g_4(\#^nK_i) = n \et c_4(\#^nK_i)\geq 2n.\]
\end{theorem}

The lower bounds for $c_4$ obtained by using Tristram-Levine signatures and the Ozsv\'{a}th-Stipsicz-Szab\'{o} $\Upsilon$-invariant fall into a rather general setup of invariants going back to Livingston~\cite{Livingston_Comp}; see Section~\ref{sec:interlude}. However, the lower bound for $c_4^\mathrm{top}$ cannot be obtained in the exact same way.
Regardless, we provide a bound for $c_4^\mathrm{top}$ (see Lemma~\ref{lem:sigbounds}) that yields the topological statements of our results above.

By design, the invariants we consider do not (or at least not in an obvious way) allow to show that $c_4(J)>2g_4(J)$ or $c_{4,+}(J)>g_4(J)$, where $c_{4,+}$ denotes a version of the $4$-dimensional clasp number that only counts positive clasps. In particular, such invariants will not allow to answer the next question.
\begin{question}\label{q:c>2g} Are there $4$-ball genus $1$ knots with arbitrarily large 4-dimensional clasp number?
Denoting by $\mathcal{C}$ the set of smooth concordance classes, is $\displaystyle\limsup_{[J]\neq 0 \in \mathcal{C}}\frac{c_4(J)}{g_4(J)}>2?$
\end{question}
In contrast, the recent work that was part of the inspiration for this note~\cite{Kronheimer-Mrowka:2019-1,Juhasz-Zemke:2020-1,Daemi-Scaduto:2020-1} has the potential to address Question~\ref{q:c>2g}. In fact, \cite[Thm.~1]{Daemi-Scaduto:2020-1} shows that the gap between $c_{4,+}$ and $g_4$ can be made arbitrarily large.

\subsection*{Acknowledgements} We thank Jennifer Hom, Allison Miller, Patrick Orson, and Mark Powell for helpful discussions.

\section{Lower bounds on $c_4$ and $g_4$ from concordance homomorphism}\label{sec:interlude} 


First, we consider the smooth category. The $4$-dimensional clasp number $c_4(J)$ of a knot $J$ can be defined as the smallest non-negative integer $k$ such that $J$ can be turned into the unknot using a finite sequence of smooth concordances and crossing changes with at most $k$ crossing changes.
Equivalently, one can define the $4$-dimensional clasp number as the smallest
non-negative integer $k$ such that there is a smoothly immersed disk with $k$ transverse double points (also called clasps) in $B^4$ bounding $J$; indeed, any such disk is isotopic to a smoothly immersed disk resulting from stacking a smooth concordance, the trace of $k$ crossing changes, and a smooth slice disk; see \cite[Prop.~2.1]{Owens-Strle:2016-1}. One defines weighted version $c_{4,+}(J)$ (resp.\ $c_{4,-}(J)$) as the smallest non-negative integer of positive-to-negative (resp.\ negative-to-positive) crossing changes in a sequence as above. In particular, for any knot $J$, 
\begin{equation}\label{eq:c+c-leqc}c_{4,+}(J)+c_{4,-}(J)\leq c_4(J).
\end{equation}
 
We recall a setup implicit in~\cite{Livingston_Comp} and explicit in \cite[Lem.~17]{Feller_14_GordianAdjacency} and \cite[Lem.~17]{LivingstonFriedlZentner}. Let $\mathcal{C}$ denote the smooth knot concordance group.

\begin{lemma}\label{lem:nu_bounds_c_4}
Let $\nu\colon\mathcal{C}\to \R$ be any homomorphism such that
\begin{itemize}
\item 
$\nu(J)\leq g_4(J)$ for every knot $J$ and
\item there exists a knot $J'$ with $\nu(J')=1$ such that $J'$ can be turned into a smoothly slice knot by a positive-to-negative crossing change,
\end{itemize}
then, for each knot $J$, we have
\begin{equation}\label{eq:nu bounds c_4}\pushQED{\qed}
|\nu(J)|\leq g_4(J),\quad\nu(J)\leq c_{4,+}(J), \et -\nu(J)\leq c_{4,-}(J).\qedhere
\end{equation}
\end{lemma}



For explicit examples of $\nu$ and their treatment via this approach, we point to \cite[Cor.~3]{Livingston_Comp} for $\nu(J)=\tau(J)$, \cite[Lem.~11 and~17]{Feller_14_GordianAdjacency} for $\nu(J)=-\sigma_{\omega}(J)/2$ when $\omega$ is regular, and \cite[Thm.~13.1]{Livingston_Upsilon} for $\nu(J)=-\Upsilon_J(t)/t$.
Here, $\omega$ is \emph{regular} if $\omega \in S^1 \smallsetminus \{ 1\} $ and  $f(\omega)\neq0$ for all integer coefficient Laurent polynomials $f$ with $f(t)=f(t^{-1})$ and $f(1)=1$.

We note that the above definitions and Lemma~\ref{lem:nu_bounds_c_4} also work in the topologically locally-flat category by replacing $\mathcal{C}$ with the topological concordance group and $g_4$, $c_{4,+}$, $c_{4,-}$, and $c_{4}$ with $g_4^\mathrm{top}$, $c_{4,+}^\mathrm{top}$, $c_{4,-}^\mathrm{top}$, and $c_{4}^\mathrm{top}$. In particular, for $\nu(J)=-\sigma_{\omega}(J)/2$ when $\omega$ is regular, \eqref{eq:nu bounds c_4} holds in the locally-flat category.

\begin{Warning}
We note the following subtlety. Above we gave two equivalent definitions of $c_4$ (and implicitly $c_{4,\pm}$) in the smooth category. It is tempting to speculate that the same equivalence of definitions holds in the locally-flat category. However, the authors are not aware of a proof of this, hence
Lemma~\ref{lem:nu_bounds_c_4} is to be read with $c_{4}^\mathrm{top}$, $c_{4,+}^\mathrm{top}$, and $c_{4,-}^\mathrm{top}$ defined using sequences of concordances and crossing changes. However, from Lemma~\ref{lem:sigbounds} below, we know that the lower bound~\eqref{eq:nu bounds c_4} for $\nu(J)=-\sigma_{\omega}(J)/2$ with $\omega$ regular
holds also when defining $c_{4}^\mathrm{top}(J)$, $c_{4,+}^\mathrm{top}(J)$, and $c_{4,-}^\mathrm{top}(J)$ 
via counting double points in locally-flat normally immersed disks filling $J$. Thus, 
Theorem~\ref{thm:claspnrvsg4fortorusknots} 
also holds for this definition of~$c_{4}^\mathrm{top}$.
\end{Warning}

We formulate the following lemma for locally-flat immersed surfaces bounding links since we believe this to be of independent interest; however, for the use in this text it suffices to consider the case of $L$ being a knot 
and $F$ being a disk in the statement (in particular, $|\eta_{\omega}(L)-b_0(F)+1|=b_1(F)=0$).
For $\omega\in S^1$, we denote by $\sigma_\omega(L)$ and $\eta_\omega(L)$ the signature and the nullity, respectively, of
$(1-\omega)M+(1-\bar\omega)M^\textrm{transpose}$, where $M$ is a Seifert matrix for $L$.
\begin{lemma}\label{lem:sigbounds}
Let $\phi\colon F\to B^4$ be a locally-flat normally immersed proper compact 
surface with $p$ positive double points and $n$ negative double points, and let $L$ be the link $\phi(\partial F)\subset S^3$. Then, for all regular
$\omega$, we have
\[\sigma_{\omega}(L)+|\eta_{\omega}(L)-b_0(F)+1|\leq b_1(F)+2n.\]
\end{lemma}
We derive Lemma~\ref{lem:sigbounds} from the Tristram-Levine bound:
\begin{equation}\label{eq:TL-bound}
|\sigma_\omega(L)|+|\eta_\omega(L)-b_0(F')+1|\leq b_1(F')\quad\text{(\cite{Tristram:1969-1,Levine:1969-1,KauffmanTaylor,Powell_17_fourgenus,ConwayNagelToffoli_17})},
\end{equation}
for every locally-flat embedded proper compact surface $F'$ with boundary the link $L\subset S^3$.
\begin{proof}[Proof of Lemma~\ref{lem:sigbounds}]
Without loss of generality we assume $F$ has no closed components (otherwise, consider the restriction of $\phi$ to the union of the non-closed components of~$F$).

We claim that there exists a locally-flat proper embedding of $F'=F\smallsetminus \{p+n\text{ open disks}\}$ into $B^4$ with boundary a link $L'$ that is the union of $L$ and $p+n$ meridians, $p$ of which are link positively with $L$ and $n$ of which link negatively with $L$. To see this claim, take $4$-balls $N_i$ around double points of $\phi(F)$ such that $(N_i,N_i\cap\phi(F))$ is  homeomorphic to $(B^4,B^4\cap \{(x,y)\in B^4\subset \C^2\mid xy=0\})$ preserving orientations, where the disk $B^4\cap \{(x,y)\mid x=0\}$ carries the orientation induced by the complex orientation and $B^4\cap \{(x,y)\mid y=0\}$ carries the induced complex and anti-complex orientation for a positive and negative double point, respectively. Choose $p+n$ properly embedded pairwise disjoint arcs $\alpha_i$ in $\phi(F)\smallsetminus N_1^\circ\cup\cdots\cup N_{p+n}^\circ$ such that $\alpha_i$ has one endpoint on $L$ and the other on $N_i$. Letting $N$ be the union of the $N_i$ and regular neighborhoods of the $\alpha_i$, we observe that the pair
$(B^4\smallsetminus N^\circ, \phi(F)\smallsetminus N^\circ)$ is homeomorphic to $(B_4, F')$ preserving orientations, where $F'$ is a locally-flat properly embedded surface with boundary a link $L'$ as claimed. In particular, $b_0(F')=b_0(F)$ and $b_1(F')=b_1(F)+p+n$.

We note that $L'$ arises as a $p+n$ fold connected sum of $L$ with $n$ positive Hopf links $H^+$ and $p$ negative Hopf links $H^-$. Hence, by additivity of signature and nullity using that $\sigma_\omega(H^\pm)=\mp 1$ and $\eta_\omega(H^\pm)=0$, we find
\begin{equation}\label{eq:LandL'}
\sigma_\omega(L')=\sigma_\omega(L)-n+p\et \eta_\omega(L')=\eta_\omega(L),
\end{equation}
 where $\eta_\omega$ and $\sigma_\omega$ denote the nullity and signature, respectively, for every $\omega\in S^1\smallsetminus \{1\}$.

We conclude the proof by calculating 
that, for every regular $\omega$, we have
\[\begin{array}{lcl}\sigma_\omega(L)+|\eta_\omega(L)-b_0(F)+1|
&\overset{\eqref{eq:LandL'}}{=}&\sigma_\omega(L')-p+n+|\eta_\omega(L')-b_0(F)+1|\\
&\overset{b_0(F')=b_0(F)}{=}&\sigma_\omega(L')+|\eta_\omega(L')-b_0(F')+1|-p+n\\
&\overset{\eqref{eq:TL-bound}}{\leq}&b_1(F')-p+n\\
&\overset{b_1(F')=b_1(F)+p+n}{=}&b_1(F)+2n.\hfill\qedhere\end{array}\]
\end{proof}

\section{Theorem~\ref{thm:claspnrvsg4fortorusknots} via signature calculations and crossing changes}\label{sec:ProofviaSig}
We start by making the family of pairs of torus knots from Theorem~\ref{thm:claspnrvsg4fortorusknots} explicit.
Recall that we write $K\coloneqq T_{p,q}\#-T_{p',q'}$.
Pairs of torus knots $\{ T_{p,q}, T_{p',q'}\}$ with Gordian distance two or more and $g_4(K)=1$ are among the following 
\begin{enumerate}[font=\upshape]
\item [(I)]\label{item:I} $\{ T_{3n+1,9n+6}, T_{3n+2,9n+3}\}$ for $n\geq1$,
\item [(II)]\label{item:II}$\{ T_{2n+1,4n+6}, T_{2n+3,4n+2}\}$ for $n\geq1$, or
\item [(III)]\label{item:III}$\{ T_{2,11}, T_{3,7}\}$,
$\{ T_{2,13}, T_{3,8}\}$,
$\{ T_{2,7}, T_{3,4}\}$,
$\{ T_{2,9}, T_{3,5}\}$,
$\{ T_{2,11}, T_{4,5}\}$,
$\{ T_{3,7}, T_{4,5}\}$,\newline
$\{ T_{3,10}, T_{4,7}\}$,
$\{ T_{4,9}, T_{5,7}\}$,
$\{ T_{3,14}, T_{5,8}\}$.
\end{enumerate}
In fact, all these pairs have Gordian distance two or more, and, except for $\{ T_{3,14}, T_{5,8}\}$, all these pairs are known to satisfy $g_4(K)=1$. See~\cite[Thm.~1.2 and Cor.~1.3]{FellerPark}.

The goal of this section is to prove Theorem~\ref{thm:claspnrvsg4fortorusknots}. The following lower bound will be used and it follows from Section~\ref{sec:interlude}.
\begin{corollary}\label{cor:signautrebound} If $J$ is a knot in $S^3$, then
\[\pushQED{\qed} c_4^\mathrm{top}(J) \geq \max_{\omega\text{ regular}}{\sigma_\omega}(K) + \max_{\omega\text{ regular}}{-\sigma_\omega}(K).\qedhere\]
\end{corollary}
\noindent The upper bound will be achieved by explicit constructions.

\subsection{Proof of Theorem~\ref{thm:claspnrvsg4fortorusknots}}
In this subsection, we provide the proof of Theorem~\ref{thm:claspnrvsg4fortorusknots} except the necessary signature calculations. The latter are provided in the next subsection.
\begin{proof}[Proof of Theorem~\ref{thm:claspnrvsg4fortorusknots}]Let $K=T_{p,q}\#-T_{p',q'}$ be such that $g_4(K)=1$ and there does not exist a crossing change turning $T_{p,q}$ into $T_{p',q'}$; in particular,
$\{T_{p,q}, T_{p',q'}\}$ must be among (I), (II), or~(III).
We establish in Examples~\ref{ex:siginfinitelist1}, \ref{ex:siginfinitelist2}, and~\ref{ex:sigfinitelist} below that
\begin{equation}\label{eq:maxsig2}
\max_{\omega\text{ regular}}{\sigma_\omega}(K)=2 \et \max_{\omega\text{ regular}}{-\sigma_\omega}(K)=2.
\end{equation}
Hence, by Corollary~\ref{cor:signautrebound} and the additivity of the Tristram-Levine signatures, we have $$c_4^\mathrm{top}(\#^n K) \geq 2n.$$ Moreover, for all integers $n\geq1$, $$g_4(\#^nK)=g_4^{\mathrm{top}}(\#^nK)=n,$$ since all these quantities are at most $n$ by $g_4(K)=1$, and the lower bound comes from the Tristram-Levine signatures bound~\eqref{eq:TL-bound}.

%


It remains to discuss the `furthermore'-paragraph. It suffices to show that for the pairs in question, $T_{p,q}$ can be turned into $T_{p',q'}$ by two crossing changes (necessarily of opposite sign by~\eqref{eq:nu bounds c_4}). Since then, $\#^nK$ can be turned into a slice knot with $2n$ crossing changes, $n$ of each sign, for all integers $n\geq1$; hence, $$2n\geq c_4(\#^nK)\geq c_4^{\mathrm{top}}(\#^nK).$$

For the pairs $\{ T_{2,7}, T_{3,4}\}$ and $\{ T_{2,9}, T_{3,5}\}$ such crossing changes are available in the literature (see e.g.\ \cite{Feller_14_GordianAdjacency}). We end the proof with an explicit sequence of two crossing changes and isotopies turning $T_{4,5}$ into $T_{3,7}$; see Figure~\ref{fig:45to37}.
\begin{figure}[h]
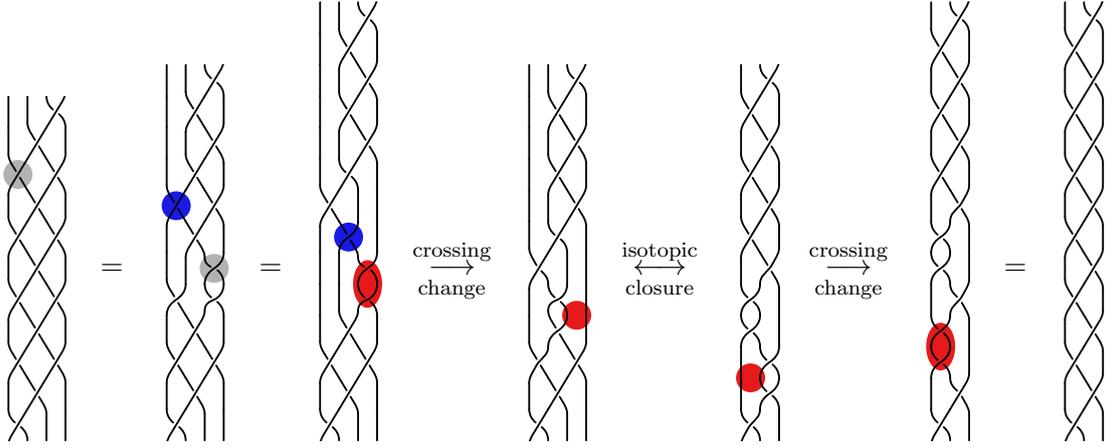

\centering
\begin{asy}
pen g = gray(0.7);
filldraw(myt * shift(1.5, 9 + 1.5) * inverse(myt) * scale(0.8) * unitcircle, g, white);
drawbraid("abcabcabcabcabc");
\end{asy}
\,\quad\raisebox{2.2cm}{$=$}
\begin{asy}
pen g = gray(0.7);
filldraw(myt * shift(3.5, 6 + 1.5) * inverse(myt) * scale(0.8) * unitcircle, g, white);
pen b = rgb(0.1, 0.1, 0.9);
filldraw(myt * shift(1.5, 8 + 1.5) * inverse(myt) * scale(0.8) * unitcircle, b, white);
drawbraid("abcabccabcabcbc");
\end{asy}
\quad\raisebox{2.2cm}{$=$}
\begin{asy}
pen b = rgb(0.1, 0.1, 0.9);
filldraw(myt * shift(2.5, 7 + 1.5) * inverse(myt) * scale(0.8) * unitcircle, b, white);
pen r = rgb(0.9, 0.1, 0.1);
filldraw(myt * shift(3.5, 5.5 + 1.5) * inverse(myt) * scale(0.8,1.3) * unitcircle, r, white);
drawbraid("abcabccbabcbcbc");
\end{asy}
\quad\raisebox{2.2cm}{$\overset{\text{crossing}}{\underset{\text{change}}{\longrightarrow}}$} 
\begin{asy}
pen r = rgb(0.9, 0.1, 0.1);
filldraw(myt * shift(3.5, 4.5 + 1.5) * inverse(myt) * scale(0.8) * unitcircle, r, white);
drawbraid("abcabbabcbcbc");
\end{asy}
\quad\raisebox{2.2cm}{$\overset{\text{isotopic}}{\underset{\text{closure}}{\longleftrightarrow}}$}
\begin{asy}
pen r = rgb(0.9, 0.1, 0.1);
filldraw(myt * shift(1.5, 2.5 + 1.5) * inverse(myt) * scale(0.8) * unitcircle, r, white);
drawbraid("abbaabababab");
\end{asy}
\quad\raisebox{2.2cm}{$\overset{\text{crossing}}{\underset{\text{change}}{\longrightarrow}}$}
\begin{asy}
pen r = rgb(0.9, 0.1, 0.1);
filldraw(myt * shift(1.5, 3.5 + 1.5) * inverse(myt) * scale(0.8,1.3) * unitcircle, r, white);
drawbraid("abaabaabababab");
\end{asy}
\,\quad\raisebox{2.2cm}{$=$}
\begin{asy}
drawbraid("ababababababab");
\end{asy}
\captionof{figure}{$T_{4,5}$ (left, as the closure of a $4$-braid) can be turned into $T_{3,7}$ (right, as the closure of a $3$-braid) by two crossing changes (modification in indicated 3-balls (red)). Also indicated ($=$) are braid isotopies (gray for the isotopy between the first and second braid, blue for the one between the second and third braid). Only the braids, rather than their closures, are drawn.}
\label{fig:45to37}
\end{figure}
 \end{proof}

\subsection{Tristram-Levine signatures calculation}
Denoting by $\sigma_t(J)=\sigma_{e^{2\pi i t}}(J)$ the Tristram-Levine signature for $\omega=e^{2\pi i t}\in S^1$ of a knot $J$~\cite{Tristram:1969-1,Levine:1969-1}, one has for every fixed knot an integer-valued piecewise linear function, which is constant in neighborhoods $t$ for which $\omega=e^{2\pi i t}$ is regular.
We define $t_0\in(0,1)$ to be a \emph{jump point} of the signature function of $J$, if the right limit $\lim_{t\to t^+}\sigma_t(J)$ differs from the left limit $\lim_{t\to t^-}\sigma_t(J)$. And say the \emph{jump} at $t_0\in(0,1)$ is $\lim_{t\to t^+}\sigma_t(J)-\lim_{t\to t^-}\sigma_t(J)$. For a more detailed discussion, and a complete description of what functions arise as signature functions, we refer to~\cite{Livingston_SigFuncs}.

For the proof of Theorem~\ref{thm:claspnrvsg4fortorusknots}, it remains to check~\eqref{eq:maxsig2} for the pairs from the families~(I), (II), and~(III).  This is done by using the signature formula going back to Hirzebruch and Brieskorn~\cite{Brieskorn_DifftopovonSing}\cite[Prop.~5.1]{GambaudoGhys_BraidsSignatures}, which we recall below in the proof of Lemma~\ref{lem:sigjumps}.
We focus on (I) and (II) since (III) consists of a small finite list of examples for which
~\eqref{eq:maxsig2} can be checked by hand or computer (see Example~\ref{ex:sigfinitelist}).
\begin{lemma}\label{lem:sigjumps}
Let $0<p<q$ be coprime integers.
The signature function
$\sigma_t(T_{p,q})$ is monotonically decreasing on $[0,\frac{p+q}{pq})$.
More precisely, the jump points on $[0,\frac{p+q}{pq})$ occur at $\{\frac{\ell}{pq}\mid p \nmid \ell \text{ and } q\nmid \ell\}$ and each jump is $-2$. 
Furthermore, the jump at $\frac{p+q}{pq}$ is $2$. In other words, for integers $\ell$ with $0\leq \ell<p+q$, we have
\[
\sigma_t(T_{p,q})=\left\{
\begin{array}{cl}
-2\left(l-\left\lfloor\tfrac{\ell}{q}\right\rfloor-\left\lfloor\tfrac{\ell}{p}\right\rfloor\right)&\text{for }t\in\left(\tfrac{\ell}{pq},\tfrac{\ell+1}{pq}\right),\vspace{.2cm}\\
-2\left(p+q-4-\left\lfloor\tfrac{q}{p}\right\rfloor\right)&\text{for }t\in\left(\tfrac{p+q}{pq},\tfrac{p+q+1}{pq}\right).\end{array}\right.\]
\end{lemma}

The content of Lemma~\ref{lem:sigjumps} might be well-known to experts. In particular, the first author previously discussed the monotonicity of the signature function of the torus knot $T_{p,q}$ on $[0,\frac{p+q}{pq})$ with Charles Livingston. For completeness, we provide a proof. 
\begin{proof}[Proof of Lemma~\ref{lem:sigjumps}] We recall the torus knot signature version of the Hirzebruch-Brieskorn signature formula~\cite{Brieskorn_DifftopovonSing} as explicitly stated in~\cite[Prop.~5.1]{GambaudoGhys_BraidsSignatures}.
For coprime integers $p,q>0$ and
\[S:=\{\tfrac{k}{p}+\tfrac{j}{q}\mid k,l\in\Z, 0<k<p \text{, and } 0<j<q\}\subset (0,2),\]
we have
\[-\sigma_{e^{2\pi i t}}(T_{p,q})=-\sigma_t(T_{p,q})=\# \{S\cap [t,1+t]\}-\#\{ S\smallsetminus (t,1+t)\}\quad \text{for all $t\in [0,1]$.}\]

The statements follow from the following 4 observations.
\begin{enumerate}
\item[i)] $\min S=\frac{p+q}{pq}$
\item[ii)] $1+\frac{p+q}{pq}\notin S$
\item[iii)] $1+\frac{\ell}{pq}\in S$ for all $0<\ell<p+q$ with $p\nmid \ell$ and $q\nmid \ell$.
\item[iv)] $\frac{\ell}{pq},1+\frac{\ell}{pq}\notin S$ for all $0<\ell<pq$ with $p\mid \ell$ or $q\mid \ell$.
\end{enumerate}
%
%


Note that i) follows immediately from the definition of $S$. Towards establishing ii)-iv), we will show that for $0<\ell<pq$, exactly one of $\frac{\ell}{pq}$ and $1+\frac{\ell}{pq}$ is in $S$ if $p\nmid \ell$ and $q\nmid \ell$, while neither is the case if $p\mid \ell$ or $q\mid \ell$. 

To see this, for $0<\ell<pq$, note that $\ell\equiv kq+jp$ mod $pq$ for some $0\leq k<p$ and $0\leq j<q$. Moreover, $p\nmid \ell$ and $q\nmid \ell$ if and only if $k\neq 0$ and $j\neq 0$, which implies iv). Since $0\leq kq+jp<2pq$, we see that $\ell$ is either $kq+jp$ or $kq+jp-pq$. If $\ell=kq+jp$ with $0< k<p$ and $0< j<q$, then $\frac{\ell}{pq} \in S$. For the sake of contradiction, assume $1+\frac{\ell}{pq} \in S$, then $pq + \ell = k'q+j'p$ for some $0< k'<p$ and $0< j'<q$. This implies that $(k'-k)q+(j'-j)p=pq$ and $p\mid (k'-k)$ and  $q\mid (j'-j)$. In particular, we have that $k=k'$ and $j=j'$, which lead us to a contradiction. A similar argument shows that if $\ell=kq+jp-pq$ with $0< k<p$ and $0< j<q$, then $1+\frac{\ell}{pq} \in S$ and $1+\frac{\ell}{pq} \notin S$.

Now, i) and iv) are already shown and we find ii) and~iii) easily as follows. ii) follows from i) and the observation above since $\frac{p+q}{pq}\in S$. iii) also  follows in a similar way since for $0<\ell<p+q$, we have that $\frac{\ell}{pq}\notin S$. 

We conclude the proof by noting that i)-iv) imply the result: iii) and iv) give that the jumps in $[0,\frac{p+q}{pq})$ are as claimed, while i) and ii) show that the jump at $\frac{p+q}{pq}$ is $2$.\end{proof}

%
%
%
%


\begin{corollary}
Let $0<p<q$ and $0<p'<q'$ be pairs of coprime integers such that $pq=p'q'$ and $p<p'$, then
\[\pushQED{\qed}
\sigma_t(T_{p,q}\#-T_{p',q'})=\left\{
\begin{array}{clr}
0&\text{for }t\in [0,\tfrac{1}{q}),\vspace{.1cm}\\
2&\text{for }t\in (\tfrac{1}{q},\tfrac{1}{q}+\tfrac{1}{pq}),\vspace{.1cm}\\
2\left(\left\lfloor\frac{p'+q'}{q}\right\rfloor+\left\lfloor\frac{p'+q'}{p}\right\rfloor-4-\left\lfloor\tfrac{q'}{p'}\right\rfloor\right)&\text{for }t\in (\tfrac{p'+q'}{qp},\tfrac{p'+q'+1}{qp}).\end{array}\right.\qedhere\]
\end{corollary}

%
%

\begin{example}\label{ex:siginfinitelist1}[Family (I)] Fix an integer $n\geq1$ and set $p=3n+1,q=9n+6,p'=3n+2,q'=9n+3$. 

\noindent If $t\in (\tfrac{1}{9n+6},\tfrac{1}{9n+6}+\tfrac{1}{3(3n+1)(3n+2)})$, then
$\sigma_t(T_{p,q}\#-T_{p',q'})=2.$

\noindent If $t\in (\tfrac{12n+5}{3(3n+1)(3n+2)},\tfrac{12n+6}{3(3n+1)(3n+2)})$, then 
\begin{align*}
\sigma_t(T_{p,q}\#-T_{p',q'})
&=2\left(\left\lfloor\tfrac{12n+5}{9n+6}\right\rfloor+\left\lfloor\tfrac{12n+5}{3n+1}\right\rfloor-4-\left\lfloor\tfrac{9n+3}{3n+2}\right\rfloor\right)\\
&=2\left(1+4-4-2\right)=-2.
\end{align*}
\end{example}

\begin{example}\label{ex:siginfinitelist2}[Family (II)]
Fix an integer $n\geq1$ and set $p=2n+1,q=4n+6,p'=2n+3,q'=4n+2$.

\noindent If $t\in (\tfrac{1}{4n+6},\tfrac{1}{4n+6}+\tfrac{1}{2(2n+1)(2n+3)})$, then
$\sigma_t(T_{p,q}\#-T_{p',q'})=2.$

\noindent If $t\in (\tfrac{6n+5}{2(2n+1)(2n+3)},\tfrac{6n+6}{2(2n+1)(2n+3)})$, then 
\begin{align*}
\sigma_t(T_{p,q}\#-T_{p',q'})
&=2\left(\left\lfloor\tfrac{6n+5}{4n+6}\right\rfloor+\left\lfloor\tfrac{6n+5}{2n+1}\right\rfloor-4-\left\lfloor\tfrac{4n+2}{2n+3}\right\rfloor\right)\\
&=2\left(1+3-4-1\right)=-2.
\end{align*}
\end{example}

\begin{example}\label{ex:sigfinitelist}[Family (III)] For all $K=T_{p,q}\#-T_{p',q'}$ from family (III), we provide some $\omega=e^{2\pi i t}$ 
for which the maximum $2$ and the minimum $-2$ of the Tristram-Levine signatures are realized.
Slightly more conceptually, we note that $pq>p'q'$ in all these examples, which immediately yields that $\displaystyle\min_{\omega\text{ regular}}{\sigma_\omega}(K)\leq-2$. The reader might argue for the maximum being $2$ using Lemma~\ref{lem:sigjumps}. However, this seems artificial given the availability of the signature functions, which in particular yield the following:
 \[\begin{array}{llcl}\textbf{$\{T_{2,11},T_{3,7}\}$:}& \sigma_t(K)=2 \quad\text{for } t\in (\tfrac{2}{21},\tfrac{3}{22}) &\text{and} & \sigma_t(K)=-2 \quad\text{for } t\in (\tfrac{1}{22},\tfrac{1}{21}),\vspace{0.1cm}\\

\textbf{$\{T_{2,13},T_{3,8}\}$:}& \sigma_t(K)=2 \quad\text{for } t\in (\tfrac{2}{26},\tfrac{3}{24}) &\text{and} & \sigma_t(K)=-2 \quad\text{for } t\in (\tfrac{1}{26},\tfrac{1}{24}),\vspace{0.1cm}\\

\textbf{$\{T_{2,7},T_{3,4}\}$:}& \sigma_t(K)=2 \quad\text{for } t\in (\tfrac{2}{12},\tfrac{3}{14}) &\text{and} & \sigma_t(K)=-2 \quad\text{for } t\in (\tfrac{1}{14},\tfrac{1}{12}),\vspace{0.1cm}\\

\textbf{$\{T_{2,9},T_{3,5}\}$:}& \sigma_t(K)=2 \quad\text{for } t\in (\tfrac{2}{15},\tfrac{3}{18}) &\text{and} & \sigma_t(K)=-2 \quad\text{for } t\in (\tfrac{1}{18},\tfrac{1}{15}),\vspace{0.1cm}\\

\textbf{$\{T_{2,11},T_{4,5}\}$:}& \sigma_t(K)=2 \quad\text{for } t\in (\tfrac{2}{20},\tfrac{3}{22}) &\text{and} & \sigma_t(K)=-2 \quad\text{for } t\in (\tfrac{1}{22},\tfrac{1}{20}),\vspace{0.1cm}\\

\textbf{$\{T_{3,7},T_{4,5}\}$:}& \sigma_t(K)=2 \quad\text{for } t\in (\tfrac{3}{20},\tfrac{4}{21}) &\text{and} & \sigma_t(K)=-2 \quad\text{for } t\in (\tfrac{1}{21},\tfrac{1}{20}),\vspace{0.1cm}\\

\textbf{$\{T_{3,10},T_{4,7}\}$:}& \sigma_t(K)=2 \quad\text{for } t\in (\tfrac{3}{28},\tfrac{4}{30}) &\text{and} & \sigma_t(K)=-2 \quad\text{for } t\in (\tfrac{1}{30},\tfrac{1}{28}),\vspace{0.1cm}\\

\textbf{$\{T_{4,9},T_{5,7}\}$:}& \sigma_t(K)=2 \quad\text{for } t\in (\tfrac{4}{35},\tfrac{5}{36}) &\text{and} & \sigma_t(K)=-2 \quad\text{for } t\in (\tfrac{1}{36},\tfrac{1}{35}), \vspace{0.1cm}\\

\textbf{$\{T_{3,14},T_{5,8}\}$:}& \sigma_t(K)=2 \quad\text{for } t\in(\tfrac{3}{40},\tfrac{4}{42}) &\text{and} & \sigma_t(K)=-2 \quad\text{for } t\in (\tfrac{1}{42},\tfrac{1}{40}).
\end{array}\]

\end{example}

\section{Theorem~\ref{thm:top} via Heegaard Floer invariants and cabling}
The goal of this section is to prove Theorem~\ref{thm:top}. We start with a proposition that summarizes the properties of the $\Upsilon$-invariant~\cite{OSS_2014} that we use in the proof.

\begin{proposition}\label{prop:upsilon_properties} For any knot $J$, the $\Upsilon$-invariant $\Upsilon_J$ is a piecewise linear function $\Upsilon_J: [0,2]\rightarrow \mathbb{R}$ with the following properties.
\begin{enumerate}[font=\upshape]
\item \cite[Cor.~1.12]{OSS_2014}\label{Upsilonadditive} $\Upsilon$ is a concordance invariant and, for all knots $J$ and $J'$,
$$\Upsilon_{J \# J'}(t) =\Upsilon_{J}(t) + \Upsilon_{J'}(t) .$$


\item \cite[Thm.~1.11]{OSS_2014}\label{Upsilongenus} For $0 < t \leq 1$, $\lvert \Upsilon_J(t)/t\rvert \leq  g_4(J).$

\item \cite[Thm.~1.14]{OSS_2014}\label{Upsilontorus} For a positive integer $i$,
\[\pushQED{\qed}
\Upsilon_{T_{2,2i+1}}(t)=
-i\cdot t \quad\mbox{for}\quad t\in [0,1].\qedhere\]

\end{enumerate}
\end{proposition}


Proposition~\ref{prop:upsilon_properties} and Lemma~\ref{lem:nu_bounds_c_4} imply the following lower bound for the $4$-dimensional clasp number. 
Compare with~\cite[Thm.~13.1]{Livingston_Upsilon} for Lemma~\ref{lem:upsilonbound} and its proof via Lemma~\ref{lem:nu_bounds_c_4}, and compare with \cite[Prop.~2.1]{Juhasz-Zemke:2020-1}, where Lemma~\ref{lem:upsilonbound} is derived from the stronger bound using $\nu^+$ given in~\cite{Hom-Wu:2016-1, Bodnar-Celoria-Golla:2017-1}. More on $\nu^+$  below.

\begin{lemma}\label{lem:upsilonbound} If $J$ is a knot in $S^3$, then
\[\pushQED{\qed} c_4(J) \geq \max_{t\in(0,1]} \Upsilon_J(t)/t + \max_{t\in(0,1]} -\Upsilon_J(t)/t .\qedhere\]
\end{lemma}

In \cite{Hom-Wu:2016-1}, Hom and Wu define a non-negative integer valued smooth concordance invariant $\nu^+$ for knots in $S^3$. Following \cite{Kim-Park:2018-1} (see also \cite{Hom:2017-1}), we say two knots $J$ and $J'$ are \emph{$\nu^+$-equivalent} if $$\nu^+(J \# - J') =\nu^+(J' \# - J)=0,$$ and it is straight forward to verify that this forms a equivalence relation on the set of concordance classes of knots. By \cite[Prop.~4.7]{OSS_2014}, we have that $\nu^+$-equivalent knots have the same $\Upsilon$-invariant. Recall that $D$ denotes the positive untwisted Whitehead double of the right-handed trefoil. The key fact that we will use is that $D$ and $T_{2,3}$ are $\nu^+$-equivalent \cite[Prop.~6.1]{Hedden-Kim-Livingston:2016-1}. Furthermore, we have the following proposition. Recall that $J_{p,q}$ denotes the $(p, q)$-cable of a knot $J$, where $p$ is the longitudinal winding.

\begin{proposition}\label{prop:nuplus} If $K_i = D_{2,2i+1} \# - T_{2,2i+1} \# -D$, then $K_i$ and $(T_{2,3})_{2,2i+1} \# - T_{2,2i+3}$ are $\nu^+$-equivalent.
\end{proposition}
\begin{proof} As mentioned above $D$ and $T_{2,3}$ are $\nu^+$-equivalent \cite[Prop.~6.1]{Hedden-Kim-Livingston:2016-1}. Furthermore, for positive integer $n$, we have that $\#^n T_{2,3}$ and $T_{2,2n+1}$ are also $\nu^+$-equivalent \cite[Thm.~B.1]{Hedden-Kim-Livingston:2016-1}.
Moreover, \cite[Thm.~B]{Kim-Park:2018-1} implies that $(T_{2,3})_{2,2i+1}$ and $D_{2,2i+1}$ are $\nu^+$-equivalent. Hence the proof is complete by noting that $\nu^+$-equivalence forms a equivalence relation on the set of concordance classes of knots and the fact that $K_1\# K_2$ and $J_1\# J_2$ are $\nu^+$-equivalent if $K_i$ and $J_i$ are $\nu^+$-equivalent for $i=0,1$ (see e.g.\ \cite[Prop.~3.12]{Kim_Krcatovich-Park:2019-1}).
\end{proof}

Finally, we compute the $\Upsilon$-invariant of $D_{2,2i+1} \# - T_{2,2i+1} \# -D$. By Proposition~\ref{prop:nuplus}, we only need to compute the $\Upsilon$-invariant of $(T_{2,3})_{2,2i+1} \# - T_{2,2i+3}$. Note that $(T_{2,3})_{2,2i+1}$ and $T_{2,2i+3}$ are both $L$-space knots and for each $L$-space knot $J$ there is a \emph{formal semigroup} $S_J$, a subset of $\Z_{\geq 0}$, associated to $J$~\cite{Wang:2018-1}. For positive integers $a_1, a_2, \ldots, a_\ell$, let
$$\langle a_1, a_2, \ldots, a_\ell \rangle \coloneqq \{c_1a_1+c_2a_2+\cdots+c_\ell a_\ell \mid c_i \in \Z_{\geq 0} \text{ for } i=1,2,\ldots, n \}.$$ For instance, the formal semigroup associated to a positive torus knot $T_{p,q}$ is $\langle p, q \rangle$. More generally, the formal semigroups of iterated cables of torus knots that are $L$-space knots can be computed. We only state the simplest case.

\begin{lemma}[{\cite[Prop.~2.7]{Wang:2018-1}}]\label{lem:semigroup} If $(T_{p,q})_{r,s}$ is an $L$-space knot, then the formal semigroup $S_{(T_{p,q})_{r,s}}$ for $(T_{p,q})_{r,s}$ is $\langle pr, qr, s\rangle$.\qed
\end{lemma}

The $\Upsilon$-invariant of an $L$-space knot can be computed in terms of its formal semigroup (see also \cite[Prop.~4.4]{Borodzik-Livingston:2016-1} for algebraic knots).

\begin{lemma}[{\cite[Prop.~3.2]{Wang:2018-1}}]\label{lem:semigroupUpsilon} If $J$ is an $L$-space knot with $g_4(J)=g$ and $S_J$ is the formal semigroup associated to $J$, then 
\[\pushQED{\qed}
\Upsilon_J(t) = \max_{m\in \{0, \ldots 2g\}} \{-2 \#\left(S_J \cap [0,m) \right) +(m-g)\cdot t\}.\qedhere\]
\end{lemma}

We need one last computational lemma before we prove Theorem~\ref{thm:top}.

\begin{lemma}\label{lem:upsiloncompute} If $K_i = D_{2,2i+1} \# - T_{2,2i+1} \# -D$ where $i>1$, then 
\[\Upsilon_{K_i}(t)=\left\{
\begin{array}{cl}
-t&\text{for }t\in [0,1/2],\vspace{0.2cm}\\
-2 +3t&\text{for }t\in [1/2,1].\end{array}\right.\]
\end{lemma}

\begin{proof} By Proposition~\ref{prop:nuplus}, $K_i$ and $(T_{2,3})_{2,2i+1} \# - T_{2,2i+3}$ are $\nu^+$-equivalent and hence have the same $\Upsilon$-invariant. Moreover, by Proposition~\ref{prop:upsilon_properties} \eqref{Upsilonadditive} and \eqref{Upsilontorus}, it suffices to compute the $\Upsilon$-invariant of $(T_{2,3})_{2,2i+1}$. For $i>1$, note that  $(T_{2,3})_{2,2i+1}$ is an $L$-space knot since $2i+1 \geq 2(2g_4(T_{2,3})-1) = 2$ \cite[Thm.~1.10]{Hedden:2009-1} and note that $g_4((T_{2,3})_{2,2i+1})=i+2$. Hence, by Lemma~\ref{lem:semigroup} and Lemma~\ref{lem:semigroupUpsilon}, we obtain
\[
\Upsilon_{(T_{2,3})_{2,2i+1}}(t) = \max_{m\in \{0, \ldots 2i+4\}} \{-2 \#\left( \langle 4, 6, 2i+1 \rangle \cap [0,m) \right) +(m-i-2)\cdot t\}.\]
We claim that for $t \in [0,1]$,
\begin{equation}\label{eq:1}
\max_{m\in \{1, \ldots 2i+4\}} \{-2 \#\left( \langle 4, 6, 2i+1 \rangle \cap [0,m) \right) +(m-i-2)\cdot t\} = -2 + (2-i)\cdot t.
\end{equation}
Indeed, if $1 \leq m \leq 3$ and $t \in [0,1]$, then 
\begin{align*}-2 \#\left( \langle 4, 6, 2i+1 \rangle \cap [0,m) \right) +(m-i-2)\cdot t\ & = -2 + (m-i-2)\cdot t \\
& \leq -2 + (2-i)\cdot t.
\end{align*}
If instead $m=4$ and $t \in [0,1]$, then 
\begin{align*}-2 \#\left( \langle 4, 6, 2i+1 \rangle \cap [0,m) \right) +(m-i-2)\cdot t\ & = -2 + (2-i)\cdot t.
\end{align*}
Finally, if $5 \leq m \leq 2i+4$ and $t \in [0,1]$, since $$\langle 4,6 \rangle = \{4,6,8,10,12,\ldots \} \subset \langle 4,6,2i+1 \rangle$$ we have 
\begin{align*}-2 \#\left( \langle 4, 6, 2i+1 \rangle \cap [0,m) \right) +(m-i-2)\cdot t\ & \leq -2 \left( \lceil m/2 \rceil -1 \right)+ \left( m-i-2 \right)\cdot t\\
&\leq -2 \lceil m/2 \rceil +2 + \left( 2\lceil m/2 \rceil-i -2 \right)\cdot t\\
&\leq -2 + (2-i)\cdot t,
\end{align*}
where the last inequality follows from the fact that if $5 \leq m \leq 2i+4$ and $t \in [0,1]$, then
$$(-2\lceil m/2 \rceil +4) \leq (-2\lceil m/2 \rceil +4)\cdot t.$$
Hence we have verified equation~\eqref{eq:1}. Finally, note that this implies that 
\begin{align*}
\Upsilon_{(T_{2,3})_{2,2i+1}}(t) &= \max_{m\in \{0, 4\}} \{-2 \#\left( \langle 4, 6, 2i+1 \rangle \cap [0,m) \right) +(m-i-2)\cdot t\}\\
&= \left\{
\begin{array}{cl}
-(i+2)\cdot t&\text{for }t\in [0,1/2],\vspace{0.2cm}\\
-2 +(2-i)\cdot t&\text{for }t\in [1/2,1].\end{array}\right.
\end{align*}
Combining the above 
with Proposition~\ref{prop:upsilon_properties}~\eqref{Upsilonadditive} and~\eqref{Upsilontorus} yields the desired formula.\end{proof}

\begin{proof}[Proof of Theorem~\ref{thm:top}] Let $K_i = D_{2,2i+1} \# - T_{2,2i+1} \# -D$ where $i>1$. The knot $D$ is a topologically slice knot~\cite{Freedman:1982-1}, which implies that $K_i$ is topologically concordant to $T_{2,2i+1} \# - T_{2,2i+1}$ which is slice.  Hence $K_i$ is topologically slice.

We claim that $g_4(\#^n K_i) =n$.
First, we show that $g_4(K_i) =1$ by following the same argument as in \cite[Lem.~3.3]{Hom-Wu:2016-1}.
We consider a genus $i+2$ Seifert surface $\Sigma$ for $D_{2,2i+1}$ obtained by taking two parallel copies of the genus one Seifert surface for $D$ and connecting them with $i$ half-twisted bands.
Note that there is a genus $i+1$ Seifert surface $\Sigma'$ for the knot $D\# T_{2,2i+1}$ embedded in $\Sigma$.
The Seifert surface $\Sigma'$ is obtained by taking the connected sum with the slightly pushed in Seifert surface for $D$ and the surface which is obtained by pushing in the half-twisted bands.
Now, we consider a Seifert surface $\widetilde{\Sigma}$ for $K_i$ obtained by taking the boundary connected sum of $\Sigma$ with the Seifert surface for $-T_{2,2i+1} \# -D$ with genus $i+1$.
Hence the genus of $\widetilde{\Sigma}$ is $2i+3$.
Moreover, taking the boundary connected sum of $\Sigma'$ with a slightly pushed in Seifert surface for $-T_{2,2i+1} \# -D$ yields a Seifert surface $\widetilde{\Sigma}'$ for $J = T_{2,2i+1} \# D \# - T_{2,2i+1} \# -D$ embedded in $\widetilde{\Sigma}$ with genus $2i+2$.
Performing surgery along $J$ oN $\widetilde{\Sigma}$ in $B^4$ yields a genus 1 surface for $K_i$. Hence we conclude that $g_4(K_i) \leq 1$.
Moreover, by Lemma~\ref{lem:upsiloncompute} and Proposition~\ref{prop:upsilon_properties} \eqref{Upsilonadditive} and  \eqref{Upsilongenus}, we have $g_4(\#^n K_i)= n$ for any positive integer $n$.

We conclude the proof by applying Lemma~\ref{lem:upsilonbound} to 
\[\Upsilon_{K_i}(t)/t = -1  \quad\mbox{for } t\in (0,1/2] \et \Upsilon_{K_i}(t)/t = 1  \quad\mbox{for } t=1\quad\text{(Lemma~\ref{lem:upsiloncompute})}.\qedhere\]
\end{proof}
\bibliographystyle{alpha}
\def\MR#1{}
\bibliography{bib}
\end{document}